\documentclass[11pt,reqno]{amsart}

\pdfoutput=1

\usepackage[letterpaper,margin=1in]{geometry}
\usepackage[backend=bibtex8,style=numeric,isbn=false,url=false]{biblatex}

\bibliography{12th-references}

\newtheorem{thm}{Theorem}

\newtheorem{lem}{Lemma}
\newtheorem{defn}{Definition}
\newtheorem{rmk}{Remark}
\newtheorem{prop}{Proposition}

\newcommand{\Mod}[1]{\ (\mathrm{mod}\ #1)}
\newcommand{\sumprime}{\mathop{\sideset{}{^*}\sum}}

\newcommand{\R}{\mathbb{R}}

\newcommand{\Z}{\mathbb{Z}}

\newcommand{\Real}{\, \mathrm{Re} \,}
\newcommand{\dx}{\, \text{d}}
\newcommand{\rt}{\text{rt}}
\newcommand\ord{\mathop{\mathrm{ord}}\nolimits}

\usepackage{changepage}
\usepackage{amssymb}
\usepackage{graphicx}

\usepackage{xcolor}
\definecolor{couleur_cite}{rgb}{0.05,.4,0.05}
\definecolor{couleur_link}{rgb}{0.05,0.05,0.4}
\definecolor{couleur_url}{rgb}{0.5,0,0}
\usepackage[hyperfootnotes=false]{hyperref}
\hypersetup{hypertexnames=false,colorlinks=true,citecolor=couleur_cite,linkcolor=couleur_link,urlcolor=couleur_url,
pdfstartview=FitH, pdfauthor=Djordje Milicevic and Daniel White, pdftitle=Twelfth moment of Dirichlet L-functions to prime power moduli}

\begin{document}

\title[Twelfth moment of $L$-functions to prime power moduli]{Twelfth moment of Dirichlet $L$-functions\\ to prime power moduli}
\author{Djordje Mili\'cevi\'c}
\address{Bryn Mawr College, Department of Mathematics, 101 North Merion Avenue, Bryn Mawr, PA 19010, USA}
\curraddr{Max-Planck-Institut f\"ur Mathematik, Vivatsgasse 7, D-53111 Bonn, Germany}
\email{dmilicevic@brynmawr.edu}
\author{Daniel White}
\address{Bryn Mawr College, Department of Mathematics, 101 North Merion Avenue, Bryn Mawr, PA 19010, USA}
\email{dfwhite@brynmawr.edu}
\thanks{D.M. was supported by the National Science Foundation, Grants DMS-1503629 and DMS-1903301.}
\begin{abstract}
We prove the $q$-aspect analogue of Heath-Brown's result on the twelfth power moment of the Riemann zeta function for Dirichlet $L$-functions to odd prime power moduli. Our results rely on the $p$-adic method of stationary phase for sums of products and complement Nunes'  bound for smooth square-free moduli.
\end{abstract}

\subjclass[2010]{Primary 11M06, 11L07; Secondary 11L40, 26E30}
\keywords{$L$-functions, moments, $p$-adic analysis, depth aspect, exponential sums, method of stationary phase}

\maketitle

\section{Introduction}
Analytic behavior of $L$-functions inside the critical strip encodes essential arithmetic information, and statistical information about their zeros, moments, and rate of growth along the critical line is of central importance in analytic number theory.
The classical Weyl bound shows that the Riemann zeta function satisfies
\begin{equation}
\label{zeta_weyl}
    \zeta \left( \tfrac{1}{2} + it \right) \ll_{\varepsilon} (1+|t|)^{1/6 + \varepsilon}
\end{equation}
where $\varepsilon > 0$ is an arbitrarily small constant that may change from one instance to another throughout this article. The widely believed Lindel\"of hypothesis asserts that $\frac16$ can be removed from the exponent above. The most recent progress in this direction is due to Bourgain \cite{Bourgain2017}, reducing the exponent to $\frac{13}{84}+ \varepsilon$. One avenue to understanding the behavior of the Riemann zeta function along the critical line is through power moments, for which asymptotic formulas are only available up to the fourth moment~\cite{Ingham1928,Motohashi1997}. Higher moments provide tighter control on large values, and in this direction Heath-Brown \cite{Heath-Brown1978} proved that, for $T\geqslant 1$,
\begin{equation}
\label{HB12}
    \int_T^{2T} \left|  \zeta \left( \tfrac{1}{2} + it \right) \right|^{12} \dx t \ll_{\varepsilon} T^{2 + \varepsilon}.
\end{equation}
This is a very elegant bound as it recovers \eqref{zeta_weyl} as a rather immediate consequence. However, \eqref{HB12} is quite a bit stronger in that it immediately implies that $\zeta \left(\frac12 + it \right)$ \emph{cannot sustain} large values; namely that
\begin{equation}
\label{HBR}
    \big|  \left\{ t \in [T,2T] : \big|\zeta\big( \tfrac{1}{2} + it\big)\big| > V  \right\}
    \big| \ll_{\varepsilon} T^{2 + \varepsilon} V^{-12}.
\end{equation}
Actually, \eqref{HB12} and \eqref{HBR} are equivalent, as is easily established via integration by parts.

Questions regarding the asymptotic behavior of $\zeta(\frac12+it)$ as $t \to \infty$ have $q$-aspect analogues concerning the central values of Dirichlet $L$-functions
$L(\frac12,\chi)$, where $\chi$ is a primitive character modulo $q$ and $q\to\infty$. For an account of some of the current literature on $L(\frac12,\chi)$ and $L$-functions in the $t$-aspect, we direct the reader to the introduction of \cite{Nunes2019}. The $q$-analogue of \eqref{zeta_weyl}, the bound $L(\frac12,\chi)\ll_{\varepsilon} q^{1/6+\varepsilon}$, long out of reach for generic $q$ except for real characters to odd square-free moduli~\cite{ConreyIwaniec2000}, has been recently announced by Petrow--Young~\cite{PetrowYoung2018}. For certain families of Dirichlet $L$-functions, however, even small improvements are known on $q^{1/6}$; see \cite{Milicevic2016} for a ``sub-Weyl'' bound $L(\frac12,\chi)\ll q^{1/6-\delta}$ for prime power moduli and \cite{Irving2016} and \cite{WuXi2016} for smooth square-free moduli.

While Dirichlet $L$-functions $L(\sigma+it,\chi)$ are also fruitfully used with a fixed modulus $q$ and large $|t|$ to study arithmetic phenomena modulo $q$, from an adelic point of view it is  more natural to consider the dependence on a large conductor $q$ as a measurement of increasing ramification, this time at finite places, and in particular, as a pure parallel to the $t$-aspect, at a fixed finite place. This explains why many tools of classical ``archimedean'' analytic number theory have found natural $p$-adic analogues. The extent of this parallel is yet to be fully understood, and our aim is to explore its manifestation for high moments of $L$-functions. Our main theorem is a $q$-aspect analogue of \eqref{HB12} for Dirichlet $L$-functions to odd prime power moduli.
\begin{thm}
\label{main}
There exists a constant $A>0$ such that, for every odd prime $p$ and every $q=p^n$,
\[ \sum_{\chi \Mod{q}} \left|L\left(\tfrac{1}{2},\chi\right)\right|^{12} \ll_{\varepsilon} p^A q^{2 + \varepsilon}. \]
\end{thm}

We remark that Theorem~\ref{main} complements the result of Nunes \cite{Nunes2019} where $q$ is taken to be smooth and square-free. The structure of the proof of Theorem~\ref{main} and the main result of Nunes translate the approach taken by Heath-Brown \cite{Heath-Brown1978} into the context of factorable and prime power moduli. For a detailed comparison between Heath-Brown's and Nunes' work, we direct the reader to the introduction of \cite{Nunes2019}. Despite the similarities, the methods of evaluation and estimation of exponential sums found throughout are quite different in the present paper. In particular, we make extensive use of a method known as $p$-adic stationary phase, which we will encapsulate in Lemmata~\ref{stationary phase} and \ref{exact stationary phase}.

As in \cite{Heath-Brown1978,Nunes2019}, the moment estimate in Theorem~\ref{main} is a consequence of the following statement, which is reminiscent of \eqref{HBR} and its relationship to \eqref{HB12}. We will establish the following.
\begin{thm}
\label{R bound}
For $V>0$, define
\[    R(V;q) := \left\{ \chi \text{ primitive of modulus } q : \left| L\left(\tfrac{1}{2},\chi\right) \right| > V \right\}. \]
Then there exists a constant $A>0$ such that, for every odd prime $p$ and every $q=p^n$,
 \[ |R(V;q)| \ll_{\varepsilon} p^Aq^{2+ \varepsilon} V^{-12}. \]
\end{thm}

Note that Theorem~\ref{main} follows immediately from Theorem~\ref{R bound} via summation by parts. From the available sharp estimates on the fourth moment of Dirichlet $L$-functions~\cite{Heath-Brown1981b,Soundararajan2007}, it follows that $|R(V;q)| \ll_{\varepsilon} q^{1 + \varepsilon} V^{-4}$; see section~\ref{mainproof_sec}. Combining this and the Weyl bound for this particular class of Dirichlet $L$-functions~\cite{Postnikov1955,Milicevic2016}, the range of interest in Theorem~\ref{R bound} is $q^{1/8-\varepsilon} \leqslant V \leqslant q^{1/6+\varepsilon}$.

The $p$-adic methods of this paper are very flexible. In particular, an analogue of \cite[Theorem 1.2]{Nunes2019}, which would sharpen Theorem~\ref{R bound} in the range $q^{3/20+\varepsilon}\leqslant V\leqslant q^{1/6-\varepsilon}$ (but not Theorem~\ref{main}) can likely be proved with a further application of the $p$-adic stationary phase method to complete exponential sums with substantially more involved phases than in \eqref{prop2-rough}. It would also be of interest to investigate whether the methods of the present paper and \cite{Nunes2019} can be unified to provide a twelfth moment bound for characters to moduli $q$ with finitely many well-located factors as in \cite{Heath-Brown1978a} (or a hybrid moment including the archimedean average); without imposing overly onerous factorization conditions, though, this may require delicate estimates on complete sums with degenerate critical points as in \cite[Lemma 7]{Heath-Brown1978a}.

\smallskip
\textbf{Overview:}
For the benefit of the reader, we present a conceptual overview of the proof, ignoring non-generic cases, coprimality conditions, $q^{\varepsilon}$ factors, and so on. We fix a divisor $q_1\mid q$,
and consider the short second moment
\begin{equation}
\label{shortmoment-def}
    S_2(\chi) := \sum_{\psi_1 \Mod{q_1}} \left| L\left(\tfrac{1}{2},\chi\psi_1\right) \right|^2.
\end{equation}
We will later choose roughly $q_1\sim V^2$, so that the expected sharp bound $S_2(\chi)\preccurlyeq q_1$ essentially matches the contribution of a single summand $|L(\frac12,\chi)|\sim V$.

Using the approximate functional equation and executing the $\psi_1$-average leads to weighted dyadic sums over $n\sim N \preccurlyeq q^{1/2}$ of terms of the form $\chi(n+hq_1)\overline{\chi}(n)$, which are $Q_1=(q/q_1)$-periodic. We apply Poisson summation, incurring the dual variable $j\preccurlyeq Q_1/N$ and the ``trace function'' $K_{\chi}(j,h;Q_1)$, which is shown in \eqref{Kchidef} and generically depends on $jh\preccurlyeq q/q_1^2$. The upshot of this analysis is Proposition~\ref{second moment bound}, which bounds $S_2(\chi)$ roughly by
\begin{equation}
\label{prop1-rough} q_1\bigg(1+Q_1^{-1/2}\sumprime_{|m|\preccurlyeq q/q_1^2}K_{\chi}(m;Q_1)A(m)\bigg),
\end{equation}
with somewhat messy arithmetic coefficients $A(m)\preccurlyeq 1$.

In Lemma~\ref{Kchi-lemma}, we show that the complete exponential sum $K_{\chi}(m;Q_1)$ exhibits square-root cancellation. This alone yields the upper bound $S_2(\chi)\preccurlyeq q_1+(q/q_1)^{1/2}$, which is sharp for $q_1\succcurlyeq q^{1/3}$ and recovers the Weyl subconvexity bound $L(\frac12,\chi)\preccurlyeq q^{1/6}$ (essentially by Weyl differencing followed by completion, as in \cite{Milicevic2016}).

For purposes of Theorems~\ref{main} and \ref{R bound}, we must consider values $q^{1/4}\preccurlyeq q_1\preccurlyeq q^{1/3}$, in which case the weighted sum of trace functions in \eqref{prop1-rough} is of length $Q_1^{1/2}\preccurlyeq q/q_1^2\preccurlyeq Q_1^{2/3}$. Weights $A(m)$ make it difficult to directly estimate the sum. Instead, the key idea is sort of a \emph{large sieve}: we argue that (roughly speaking, and as $q_1$ gets smaller) the vectors $(K_{\chi}(m;Q_1))_m$ are typically approximately orthogonal for different $\chi$, and thus it is hard for too many of them to avoid cancellation with a single vector $(A(m))_m$. The approximate orthogonality boils down to cancellations in \emph{incomplete sums of products}; since the length is over the square-root of the conductor, we apply the method of completion, incurring an additive twist. Proposition~\ref{bound for completion}, our key arithmetic input, shows square-root cancellation in sums of products of rough form
\begin{equation}
\label{prop2-rough}
\sumprime_{u\Mod{Q}}K_{\chi}^{\pm}(u;Q_1)\overline{K_{\chi'}^{\pm}(u;Q_1)}e(-uv/Q)\preccurlyeq Q^{1/2}.
\end{equation}
Here, the modulus $Q\mid Q_1$ drops with the conductor of $\chi\overline{\chi'}$ (essentially the distance between $\chi$ and $\chi'$ in the dual topology), and we must first separate $K_{\chi}$ into two oscillatory components $K_{\chi}^{\pm}$ (as often happens with Bessel functions; see also \cite[\S 9]{BlomerMilicevic2015}). Lemma~\ref{Kchi-lemma} and Proposition~\ref{bound for completion} form the heart of the paper and are proved by a consistent application of the \emph{$p$-adic method of stationary phase} to exponential sums with $p$-adically analytic phases, including characters to prime power moduli; see section~\ref{prelim-section}.

Proceeding with the large sieve idea, we estimate the the sum of $S_2(\chi\psi)$ in \eqref{prop1-rough} over an arbitrary set $\Psi$ of characters $\psi$ modulo some $q_2\mid q$ (with $q_1\mid q_2$) by applying the Cauchy--Schwarz inequality to the $m$-sum and bounding sums of products of $K_{\chi\psi}^{\pm}(m;Q_1)$ using \eqref{prop2-rough}. This shows in Proposition~\ref{short moment bound} that
\begin{equation}
\label{prop3-rough}
\sum_{\psi\in \Psi} S_2 (\chi \psi) \preccurlyeq \left(  \big(q_1 + q_1^{1/4}q_2^{1/4} \big) |\Psi| + q^{1/2} |\Psi|^{1/2} \right).
\end{equation}
The bound \eqref{prop3-rough} imposes a restriction on the size $|\Psi|$ as long as each $S_2(\chi\psi)$ is slightly bigger than $q_1+q_1^{1/4}q_2^{1/4}$. In section~\ref{mainproof_sec}, we first fix $\chi$ and choose $\Psi$ to be the set of characters modulo $q_2$ for which one of $|L(\frac12,\chi\psi\psi_1)|$ in \eqref{shortmoment-def} exceeds $V$, with $q_1 = q^{-\varepsilon} V^2$ and $q_2=q_1^3$, obtaining $|\Psi|\preccurlyeq qV^{-4}$. From here it is a matter of bookkeeping to Theorem~\ref{R bound} and hence Theorem~\ref{main}.

\smallskip
\textbf{Notation:} Throughout the paper, $\varepsilon>0$ indicates a fixed positive number, which may be different from line to line but may at any point be taken to be as small as desired. As usual, $f\ll g$ and $f=O(g)$ indicate that $|f|\leqslant Cg$ for some effective constant $C>0$, which may be different from line to line but does not depend on any parameters except as follows. In this introduction, all implied constants in $\ll$ and $O$ are absolute, except that they may depend on $\varepsilon>0$ if so indicated as in $\ll_{\varepsilon}$. In the rest of the paper, we allow the implied constants (but suppress this from notation) to depend on both the odd prime $p$ and $\varepsilon>0$. All dependencies on $p$ are easily seen to be polynomial, leading to the statements of Theorems~\ref{main} and \ref{R bound}; we do not make an effort to optimize the value of $A>0$. Finally, in the informal outline in the introduction only, we also use $f\preccurlyeq g$ to denote $|f|\ll_{p,\varepsilon}q^{\varepsilon}g$ and $f\sim g$ for $f\preccurlyeq g\preccurlyeq f$.

We denote the cardinality of a finite set $S$ by $|S|$; we use the same notation for the Lebesgue measure, with the meaning clear from the context. As is customary in analytic number theory, we also write $e(z)=e^{2\pi iz}$.

\textbf{Acknowledgements:} The authors would like to thank an anonymous referee for their careful reading and constructive suggestions, which helped us improve the paper in several places.

\section{Preliminaries}
\label{prelim-section}

\subsection{Approximate functional equation}
A ubiquitous tool in the analysis of $L$-functions inside the critical strip
is the approximate functional equation (see \cite[\S 5.2]{IwaniecKowalski2004}). This equation has various manifestations depending on context and purpose. A typical form of this equation in the context of bounding central values states that one may recover the size of $L(\frac12,\chi)$ by inserting $s = \frac12$ into the associated Dirichlet series which is essentially truncated at $q^{1/2}$ via a suitable smooth weight function. For our purposes, the following lemma is convenient, which follows by applying a dyadic partition of unity to \cite[Theorem~5.3]{IwaniecKowalski2004}.
\begin{lem}
\label{AFE2}
Let $\chi$ be a primitive Dirichlet character modulo $q$. Then,
\[
    \big|L\big(\tfrac{1}{2},\chi\big)\big|^2 \ll \log q \sum_{\substack{N \leqslant q^{1/2 + \varepsilon} \\ N \textnormal{ dyadic}}} \left| \frac{1}{\sqrt{N}} \sum_n \chi(n) V_N(n) \right|^2+ q^{-100},
\]
where $V_N$ is a smooth function depending only on $N$ and $q$, whose support is contained in $[N/2,2N]$ and whose derivatives satisfy $V_N^{(j)} \ll_j N^{-j}$ for every $j\in\mathbb{N}$.
\end{lem}

\subsection{\texorpdfstring{$p$-adically}{p-adically} analytic phases}
Among the key features of our treatment of exponential sums will be: \textit{(i)} the consistent interpretation of oscillating terms (such as characters) as exponentials with phases that are $p$-adically analytic functions and \textit{(ii)} the analysis thereof. For a rigorous treatment of these concepts, we refer to \cite[\S 2]{Milicevic2016}. Recall that a $p$-adically analytic function $f$ on a domain $D\subseteq\mathbb{Z}_p$ is locally expressible, around each point $a\in D$, in a $p$-adic ball of the form $\{x\in\mathbb{Z}_p:|x-a|_p\leqslant p^{-\varrho}\}\subseteq D$ ($\varrho\in\mathbb{Z}_{\geqslant 0}$) as the sum of its $p$-adically convergent Taylor power series. We let $r_p(f;a)$ denote the largest such $p^{-\varrho}$ (which is not quite the same as the $p$-adic radius of convergence) and $r_p(f)=\inf_{a\in D} r_p(f;a)\geqslant 0$; in all phases we will encounter, $r_p(f)\geqslant p^{-1}$ will hold. It is not hard to see that $r_p(f';a)\geqslant r_p(f;a)$.

We will make extensive use of the
$p$-adic logarithm, which for simplicity we define on $1+p\mathbb{Z}_p$. Recall that, throughout the paper, $p$ is an odd prime.
\begin{defn}
The $p$-adic logarithm, $\log_p : 1+p\Z_p \to p\Z_p$ is the analytic function given as
\[
\log_p (1 + x) := \sum_{k \geqslant 1} (-1)^{k-1} \frac{x^k}{k}.
\]
\end{defn}

Access to the above is critical due to the following lemma, with roots in Postnikov \cite{Postnikov1955} and which we quote from \cite[Lemma~13]{Milicevic2016}.
\begin{lem} 
\label{Postnikov}
Let $\chi$ be a primitive character modulo $p^n$. Then there exists a $p$-adic unit $A$ such that, for every $p$-adic integer $k$,
\begin{equation}
\label{Postnikov formula}
\chi(1 + kp) = e \left( \frac{A \log_p (1+kp)}{p^n} \right).
\end{equation}
\end{lem}
Lemma~\ref{Postnikov} allows us to explicate the phase of any exponential of the form $\chi(1+kp) e(f(k)/p^n)$ when $\chi$ is a character modulo $p^n$.

It will be necessary to handle solutions to quadratic equations over $\Z_p$, which requires the use of $p$-adic square roots. For $p$ an odd prime and $x \in \Z_p^{\times \, 2}$, the congruence $u^2 \equiv x \Mod{p^{\kappa}}$ has exactly two solutions modulo every $p^{\kappa}$, which reside within two $p$-adic towers and limit to the solutions of $u^2=x$ as $\kappa \to \infty$. We denote these solutions $\pm x_{1/2}$. For $(\, \cdot \,)_{1/2} :\Z_p^{\times 2}\to\Z_p^{\times}$ to be well-defined, a choice of square root for each $y\in (\Z / p \Z)^{\times 2}$ must be made. This set of choices propagates to $\Z_p^{\times 2}$ and represents one of the $2^{(p-1)/2}$ branches of the $p$-adic square root. A thorough treatment of $p$-adic square roots can be found in \cite[\S 2]{BlomerMilicevic2015}; we content ourselves with summarizing two properties of import to us.

Each branch $x_{1/2}$ of the square root is an analytic function expressible by a convergent power series in balls of radius $r_p\geqslant p^{-1}$. Specifically, on $1+p\Z_p$, the binomial expansion
\begin{equation}
\label{sqrt-expansion}
(1 + xp)^{1/2} = \sum_{k \geqslant 0} \binom{1/2}{k} (xp)^k
\end{equation}
gives the branch with values in $1+p\Z_p$ (as seen by formally squaring the right-hand side), which is in fact an automorphism of $1+p\Z_p$. For an arbitrary $u\in\Z_p^{\times 2}$, a simple argument modulo $p$ shows that
\begin{equation}
\label{sqrt-mult}
(u + xp)_{1/2} = u_{1/2} (1 + x \overline{u}p)^{1/2}
\end{equation}
where $\overline{u}$ denotes the $p$-adic inverse of $u$. While $( \, \cdot \,)_{1/2}$ cannot in general be expected to be multiplicative, \eqref{sqrt-mult} gives it both a pseudo-morphism rule and a power expansion. Moving forward, we fix a branch to be used throughout, drop the $( \, \cdot \, )_{1/2}$ notation and simply write $(\, \cdot \,)^{1/2}$ or use a radical symbol for our chosen branch, using caution to only use \eqref{sqrt-expansion}, \eqref{sqrt-mult}, and $\sqrt{m}^2=m$
when exercising the usual archimedean exponent rules. For future reference, we note that, for all $u,u'\in\Z_p^{\times 2}$,
\begin{equation}
\label{ord-sqrt}
\ord_p(\sqrt{u}-\sqrt{u'})=\ord_p(u-u').
\end{equation}

\subsection{\texorpdfstring{$p$-adic}{p-adic} method of stationary phase}
The following pair of lemmata establishes what is known as the $p$-adic method of stationary phase (see, for example, \cite[\S 4]{Milicevic2016}, \cite[\S 7]{BlomerMilicevic2015}), allowing one to evaluate complete sums involving such exponentials. They are the proper $p$-adic analogues of the classical method of stationary phase for exponential integrals of the form $\int_{\mathbb{R}}g(x)e(f(x))\,\mathrm{d}x$ with a suitable smooth phase $f$ and weight $g$, which generically proceeds in two principal steps: \textit{(i)} showing that ranges where $|f'|$ is not suitably small are negligible, and \textit{(ii)} close to each non-degenerate stationary point $x_0$ of the phase $f$, approximating $f$ quadratically, with resulting Gaussian-type integrals evaluating to about $g(x_0)e(f(x_0))/\sqrt{|f''(x_0)|}$ (see~\cite{GrahamKolesnik1991}).

\begin{lem}
\label{stationary phase}
Let $p$ be an odd prime, $1 \leqslant \ell \leqslant n$ be integers, and $f: \Z_p^{\times} \to \Z_p$ be an analytic function invariant modulo $p^n$ under translation by $p^n \Z_p$. If $r_p(f)\geqslant p^{-\ell}$ and $p^{k \ell }f^{(k)}(x)/k! \equiv 0 \Mod{p^n}$ for all $x \in \Z_p^{\times}$ when $k \geqslant 2$, then
\[     \sumprime_{x \Mod{p^n}} e \left(   \frac{f(x)}{p^n} \right) = \sumprime_{\substack{x_0 \Mod{p^n} \\ f'(x_0) \equiv 0 \Mod{p^{n-\ell }}}} e \left( \frac{f(x_0)}{p^n} \right). \]
\end{lem}

\begin{proof}
Expanding $f(x)$ around $x_0$ gives $f(x_0 + t p^{\ell}) = \sum_{k \ge 0} f^{(k)}(x_0) (t p^{\ell})^k/k!$. With this, observe
\[   \sumprime_{x \Mod{p^n}} e \left(   \frac{f(x)}{p^n} \right) = \frac{1}{p^{n-\ell}} \sumprime_{x_0 \Mod{p^n}} \sum_{t \Mod{p^{n-\ell}}} e \left( \frac{f(x_0) + f'(x_0) t  p^{\ell}}{p^n} \right) \]
where the inner sum contributes $p^{n-\ell }e(f(x_0)/p^n)$ when $f'(x_0) \equiv 0 \Mod{p^{n- \ell }}$ and vanishes otherwise.
\end{proof}

Lemma~\ref{stationary phase} reduces a complete exponential sum to $p$-adic neighborhoods in which $|f'(x)|_p$ is small. The following lemma is a further refined statement that explicitly evaluates these localized sums and is suited for exponential sums that we will encounter in the proof of Lemma~\ref{Kchi-lemma}.

\begin{lem}
\label{exact stationary phase}
Let $p$ be an odd prime, $n\geqslant 2$, and $f: \Z_p^{\times} \to \Z_p$ be an analytic function satisfying the hypotheses in Lemma~\ref{stationary phase} for $\ell = \lceil n/2 \rceil$. Let $X\subseteq(\mathbb{Z}/p^n\mathbb{Z})^{\times}$ denote the solution set of $f'(x_0)\equiv 0\pmod{p^{\lfloor n/2\rfloor}}$, and assume that, for all $x_0\in X$, $r_p(f;x_0)\geqslant p^{-\lfloor n/2 \rfloor}$, $f''(x_0)\in\mathbb{Z}_p$, and $p^{ \lfloor n/2 \rfloor k} f^{(k)}(x_0)/k! \equiv 0 \Mod{p^n}$ for $k \geqslant 3$. Then, $X$ is invariant under translation by $p^{\lfloor n/2\rfloor}\mathbb{Z}$, and, for an arbitrary set of representatives $\tilde{X}$ for $X/p^{\lfloor n/2\rfloor}\mathbb{Z}$,
\[    \sumprime_{x \Mod{p^n}} e \left( \frac{f(x)}{p^n} \right) = p^{n/2}\sum_{x_0 \in\tilde{X}}  e \left( \frac{f(x_0) }{p^n} \right) \Delta_f (x_0;p^n), \]
where all summands are independent of the choice of $\tilde{X}$, and, writing $f'(x_0)_{\circ}:=f'(x_0)/p^{\lfloor n/2\rfloor}$ and $\big( \tfrac{\cdot}{p} \big)$ for the Legendre symbol,
\[ 
\Delta_f(x_0;p^n) =
     \begin{cases}
       1, &  2\mid n;\\
       \epsilon(p) \big( \frac{2f''(x_0)}{p} \big)e \left( \frac{-\overline{2f''(x_0)}f'(x_0)_{\circ}^2}{p} \right), & 2\nmid n,p\nmid f''(x_0);\\
	\sqrt{p}\mathbf{1}_{p\mid f'(x_0)_{\circ}}, &2\nmid n,p\mid f''(x_0),
     \end{cases}
\quad
\epsilon(p)=\begin{cases}
1, & p \equiv 1 \Mod{4};\\
i, & p \equiv 3 \Mod{4}.
\end{cases}
\]
\end{lem}

\begin{proof}
The translational invariance of $X$ modulo $p^{\lfloor n/2\rfloor}\mathbb{Z}$ is clear from our hypotheses and the expansion of $f'(x_0+tp^{\lfloor n/2\rfloor})$ at each $x_0\in X$. Application of Lemma~\ref{stationary phase} with $\ell = \lceil n/2 \rceil$ together with an expansion of $f$ around each $x_0 \in \tilde{X}$ gives
\begin{align*}
    \sumprime_{x \Mod{p^n}} e \left( \frac{f(x)}{p^n} \right) & = \sum_{x_0 \in \tilde{X}} \sum_{t \Mod{p^{\lceil n/2 \rceil}}} e \left( \frac{f(x_0 + t p^{\lfloor n/2 \rfloor})}{p^n} \right) \\
    &= p^{\lfloor n/2 \rfloor} \sum_{x_0 \in \tilde{X}} e \left( \frac{f(x_0)}{p^n} \right) \sum_{t \Mod{p^{ n - 2 \lfloor n/2 \rfloor }}} e \left( \frac{f'(x_0)_{\circ} t + \bar{2}f''(x_0) t^2}{p^{n - 2 \lfloor n/2 \rfloor}} \right).
\end{align*}
For $n$ even, the inner sum is trivial and the desired result follows. If $n$ is odd, the contribution from $p\mid f''(x_0)$ is clear, while, for $p\nmid f''(x_0)$, completing the square yields for the inner sum
\[ e\bigg(\frac{-\overline{2f''(x_0)}f'(x_0)_{\circ}^2}{p} \bigg) \sum_{t \Mod{p}} e \left( \frac{\bar{2}f''(x_0)t^2}{p}  \right)
=\epsilon(p) \sqrt{p} \left( \frac{\bar{2}f''(x_0)}{p} \right)e\bigg(\frac{-\overline{2f''(x_0)}f'(x_0)_{\circ}^2}{p} \bigg), \]
by the classical evaluation of the quadratic Gauss sum. This finishes the proof.
\end{proof}

\begin{rmk}
\label{simplifies}
The general (if somewhat cumbersome) conditions in
Lemma \ref{exact stationary phase} are easily satisfied, say, for every analytic function $f: \Z_p \to \Z_p$ with $r_p(f)\geqslant 1/p$ and $f^{(j)}(\mathbb{Z}_p)\subseteq\mathbb{Z}_p$ for all $j\geqslant 0$. The same is true for Lemma~\ref{stationary phase} with $\ell\geqslant n/2$.

In Lemma~\ref{exact stationary phase}, in the odd nonsingular case $2\nmid n$, $p\nmid f''(x_0)$, we see that $f'(x_0+tp^{\lfloor n/2\rfloor})\equiv 0\pmod {p^{\lceil n/2\rceil}}$ for exactly one $t\pmod p$; picking such a representative $\tilde{x}_0:=x_0+tp^{\lfloor n/2\rfloor}\in\tilde{X}$, we have more simply $\Delta_f(\tilde{x}_0;p^n)=\epsilon(p)\big(\frac{2f''(\tilde{x}_0)}{p}\big)$.
\end{rmk}

\begin{rmk}
\label{invariance}
Versions of Lemmata~\ref{stationary phase} and \ref{exact stationary phase} exist for sums over other subsets of residue classes $x \Mod{p^n}$, where the phase $f$ may have as domain a finite union of translates of $p^{\lambda}\Z_p$, with $\lambda\leqslant\ell$ in Lemma~\ref{stationary phase} and $\lambda\leqslant n/2$ in Lemma~\ref{exact stationary phase}.
The proofs of these parallel statements are the same; indeed, they only require that the sum be over a set of residues invariant under translation by a suitable $p^{\ell}\Z_p$ with $\ell\geqslant\lambda$ and $r_p(f)\geqslant p^{-\ell}$.
Specifically, the proof of Proposition~\ref{bound for completion} will require Lemma~\ref{stationary phase} to be applied over quadratic residues and non-residues modulo $p^n$. Lemmata~\ref{stationary phase} and \ref{exact stationary phase} also hold for sums of the form $\sum g(x) e(f(x)/p^n)$ where $g$ is invariant under translation by $p^{\ell} \Z_p$.
\end{rmk}

In practice, we will apply Lemmata~\ref{stationary phase} and \ref{exact stationary phase} in situations where explicitly writing the exponent of $q=p^n$ gets notationally cumbersome. To represent what are essentially square roots in these cases, we define
\[ \rt_*(p^n) := p^{\lfloor n/2 \rfloor} \quad \text{and} \quad \rt^*(p^n) := p^{\lceil n/2 \rceil}. \]

\subsection{Completion}
While Lemmata~\ref{stationary phase} and \ref{exact stationary phase} provide powerful tools for evaluating the types of complete exponential sums that will be found throughout, we will eventually encounter those which are incomplete. In anticipation of this, we introduce the next lemma which prepackages a technique known as completion. 

\begin{lem}
\label{completion}
Suppose $f$ is an arithmetic function with period $Q$. Then
\[ \sum_{m \leqslant M} f(m) \ll \frac{1}{Q} \sum_{v \Mod{Q}} \big|\widehat{f}(v)\big| \cdot \min \left\{ M, \| v/Q \|^{-1} \right\},\quad
\widehat{f}(v):= \sum_{u \Mod{Q}} f(u) e \left( \frac{-uv}{Q} \right), \]
where $\|x\|$ is the distance from $x$ to the nearest integer.
\end{lem}

\begin{proof}
Splitting the sum into residue classes modulo $Q$ yields
\[    \sum_{m \leqslant M} f(m) = \sum_{u \Mod{Q}} f(u) \sum_{m \leqslant M} \frac{1}{Q} \sum_{v \Mod{Q}} e \left( \frac{(m-u)v}{Q} \right) \\
    = \frac{1}{Q} \sum_{v \Mod{Q}} \widehat{f}(v) \sum_{m \leqslant M} e \left( \frac{mv}{Q} \right). \]
The bound
\[    \sum_{m \leqslant M} e \left( \frac{mv}{Q} \right) \ll \min \{ M, \|v/Q\|^{-1} \} \]
on the sum of a geometric sequence completes the proof.
\end{proof}

\section{Short second moment}
As before and throughout, $p$ will be an odd prime and $q$ some prime power $p^n$ for $n$ a positive integer. Further consider
\begin{equation}
\label{q restrictions}
p \leqslant q_1 \leqslant q_2 < q,    
\end{equation}
where the $q_i$ are also powers of $p$. A central object to our proofs, as in \cite{Heath-Brown1978,Nunes2019},
is the short second moment. In the $q$-aspect, this will be a power moment which samples from a $q_1$-neighborhood around some fixed primitive character $\chi \Mod{q}$. This analogy is particularly natural from a $p$-adic point of view, as the (Pontrjagin) dual group of $\mathbb{Z}_p^{\ast}$ carries the natural dual topology, with respect to which these correspond to actual small neighborhoods of $\chi$. We denote
\begin{equation*}
    S_2(\chi) := \sum_{\psi_1 \Mod{q_1}} \left| L\left(\tfrac{1}{2},\chi\psi_1 \right) \right|^2.
\end{equation*}
We will eventually analyze the size of short moments on average, but first must gather information on $S_2(\chi)$ itself.

\subsection{Executing the short second moment} 
We immediately apply Lemma~\ref{AFE2} to $S_2(\chi)$. This yields
\begin{align}
    S_2(\chi) & \ll q^{\varepsilon} \sum_{\psi_1 \Mod{q_1}} \sum_{\substack{N \leqslant q^{1/2 + \varepsilon}\\N \text{ dyadic}}} \left| \frac{1}{\sqrt{N}} \sum_n \chi \psi_1(n) V_N(n)\right|^2 + q^{-99} \nonumber \\
    & \label{after AFE} \ll \frac{q^{\varepsilon}}{N} \sum_{\psi_1 \Mod{q_1}} \left| \sum_n \chi \psi_1(n) V_N(n) \right|^2 + q^{-99}
\end{align}
for some $N \leqslant q^{1/2 + \varepsilon}$ by exchanging order of summation and choosing the summand which maximizes the inner sum. Denote the sum in \eqref{after AFE} without the error term and $q^{\varepsilon}$ factor as $B(N)$. Expansion and orthogonality of characters gives
\[    B(N) \ll \frac{q_1}{N} \sum_{n \equiv n' \Mod{q_1}} \chi(n') \overline{\chi}(n) V_N(n') \overline{V_N(n)}. \]
We note the similarity of the resulting sum (the sum of squares of short $p$-adic averages, a reflection of the $\psi_1$-average via Parseval's identity) to those encountered with Weyl differencing in the context of factorable moduli (see, for example, \cite[\S 5]{Milicevic2016}), and we proceed similarly.

Recall that $V_N \ll 1$ with support contained in $[N/2,2N]$. The diagonal terms corresponding to $n=n'$ contribute $O(q_1)$ to $B(N)$. The addition of this to the remaining pairs $(n',n)$ gives
\begin{equation}
\label{after diagonal}
    B(N) \ll q_1 + \frac{q_1}{N} \Real \sum_{h \ge 1} \sum_{n \ge 1} \chi (n + h q_1) \overline{\chi}(n) V_N(n + h q_1) \overline{V_N(n)},
\end{equation}
since each $(n',n)$ appears above or can be accounted for by conjugation. Denote the inner sum in \eqref{after diagonal} as $S_{hq_1}(N;\chi)$. Since $\chi(n + h q_1) \overline{\chi}(n)$ is periodic modulo $Q_1 = q/q_1$, we may write
\begin{equation}
    \label{before transform}
    S_{hq_1}(N; \chi) = \sumprime_{r \Mod{Q_1}} \chi(r + hq_1) \overline{\chi}(r) \sum_j V_N(r + jQ_1 + hq_1) \overline{V_N(r + jQ_1)}.
\end{equation}
We will apply Poisson summation to the inner sum in $S_{h q_1}$. Examination of
\begin{align}
    \label{transform} &\int_{\R} V_N(r + xQ_1 + hq_1)\overline{V_N(r+xQ_1)} e(-jx) \dx x \\
    &\qquad= Q_1^{-1} e(rj/Q_1) \int_{\R} V_N(y+hq_1) \overline{V_N(y)} e(-jy/Q_1) \dx y \nonumber
\end{align}
shows that the Fourier transform in \eqref{transform} is
\begin{equation}
    \label{transform2}
    Q_1^{-1} e \left( \frac{rj}{Q_1} \right)  \widehat{W_{hq_1}} (j/Q_1) \quad \text{where} \quad W_{hq_1}(y) := V_N(y + hq_1) \overline{V_N(y)}.
\end{equation}
Using \eqref{after AFE} through \eqref{transform2} together with Poisson summation yields
\begin{equation}
    \label{after Poisson}
    S_2(\chi) \ll q^{\varepsilon} \bigg( q_1 + \frac{q_1}{N} \Real \sum_{h \ge 1} Q_1^{-1/2} \sum_j \widehat{W_{hq_1}}(j/Q_1) K_{\chi}(j,h;Q_1) \bigg),
\end{equation}   
where, for $\tilde{q}$ a proper divisor of $q$, we define
\begin{equation}
\label{Kchidef}
K_{\chi}(a,b;\tilde{q}) := \tilde{q}^{-1/2} \sumprime_{r \Mod{\tilde{q}}} \chi\big(r+b(q/\tilde{q})\big)\overline{\chi}(r) e(ar/ \tilde{q}).
\end{equation}
In particular, for $(m,\tilde{q})=1$, we also write $K_{\chi}(m;\tilde{q}):=K_{\chi}(1,m;\tilde{q})=K_{\chi}(m,1;\tilde{q})$, so that
\begin{equation}\label{Kchi-m}
K_{\chi}(m;\tilde{q})=\tilde{q}^{-1/2}\sumprime_{r\Mod{\tilde{q}}}\chi\big(1+(q/\tilde{q})r\big)e(m\bar{r}/\tilde{q}).
\end{equation}
This sum (which takes on the role of trace functions from the context of square-free moduli~\cite{Nunes2019}) is of central importance to our arguments. We summarize some of its important properties in Lemma~\ref{Kchi-lemma} in section~\ref{expsums-sec}, below. In this section, we will only require the elementary reduction and vanishing claim \eqref{Kchi-reduction}.

By the support of $V_N$, we may actually take $h \ll N/q_1$ in \eqref{after Poisson}. We will soon identify the range $j$ that is essential to \eqref{after Poisson}. Once this range becomes finite, we will configure our bound in a way that highlights the main object of our study.

\subsection{Establishing the bound on \texorpdfstring{$S_2(\chi)$}{S2(chi)}}
We first show that the contribution to \eqref{after Poisson} from $j=0$ may be neglected. By Lemma~\ref{Kchi-lemma} below,
\[    K_{\chi}(0,h;Q_1) = \begin{cases} Q_1^{1/2} (1-p^{-1}),
&Q_1 \mid h;\\ -Q_1^{1/2}/p, & Q_1/p\mathrel{\|}h;\\
0, & \text{otherwise}.\end{cases} \]
The contribution from $j=0$ to \eqref{after Poisson} is then
\begin{equation}
\label{j=0}
    q^{\varepsilon} q_1 Q_1^{-1/2} \Real \sum_{1 \leqslant h \ll N q_1^{-1}} N^{-1} \widehat{W_{hq_1}}(0) K_{\chi}(0,h;Q_1) \ll q^{\varepsilon} q_1 \sum_{\substack{ 1 \leqslant h \ll N q_1^{-1} \\ Q_1 \mid hp}}
    1
    \ll q^{\varepsilon}
    \frac{N}{Q_1}.
\end{equation}
Repeated use of integration by parts shows
\[ \widehat{W_{hq_1}}(y) \ll_m N \left( \frac{1}{Ny} \right)^m \]
for every positive integer $m$. From this bound, $h\ll N/q_1$, and the trivial bound on $K_{\chi}(j,h;Q_1)$, the contribution to \eqref{after Poisson} from $|j|>q^{\varepsilon}Q_1/N$ is $O(q^{-100})$. Using \eqref{after Poisson} and \eqref{j=0}, we find
\begin{equation}
\label{reduced range}
    S_2(\chi) \ll q^{\varepsilon} \bigg( q_1 + \frac{q_1}{Q_1^{1/2}} \Real \sum_{0 < h \ll N/q_1} \sum_{0 < |j| \ll q^{\varepsilon}Q_1/N} N^{-1} \widehat{W_{hq_1}}(j/Q_1) K_{\chi}(j,h;Q_1) \bigg).
\end{equation}
According to Lemma~\ref{Kchi-lemma}, noting that $(Q_1^2/p)\nmid jh\ll q^{\varepsilon}Q_1/q_1$, we may rewrite the double sum above as
\begin{align}
    &\sum_{p^{\eta} \ll q^{1/2 + \varepsilon}q_1^{-1}}\sumprime_{h' \ll N(q_1p^{\eta})^{-1}} \sumprime_{0 < |j'| < q^{\varepsilon} Q_1 (Np^{\eta})^{-1}} N^{-1} \widehat{W_{h'p^{\eta}q_1}} (j'p^{\eta}/Q_1) K_{\chi}(j'p^{\eta},h'p^{\eta};Q_1) \nonumber \\
    &\quad
    \label{factor m}
    = \sum_{p^{\eta} \ll q^{1/2 + \varepsilon}q_1^{-1}} p^{\eta/2} \sumprime_{|m|
    \ll q^{1 + \varepsilon} (q_1^2 p^{2 \eta})^{-1}}
    K_{\chi}(m
    ;Q_1/p^{\eta})\sum_{\substack{h'|j'| = m 
    \\ h' \ll N(q_1 p^{\eta})^{-1} \\ |j'| < q^{\varepsilon} Q_1 (Np^{\eta})^{-1}}} N^{-1} \widehat{W_{h' p^{\eta} q_1}}(j' p^{\eta}/ Q_1).
\end{align}
Denoting the inner sum of \eqref{factor m} as $A(m;p^{\eta})$, the above becomes
\begin{equation}
\label{K alpha}
    \sum_{p^{\eta}\ll q^{1/2+ \varepsilon}q_1^{-1}}p^{\eta/2}\sumprime_{\substack{|m| \ll q^{1 + \varepsilon} (q_1^{2}p^{2\eta})^{-1}}}  K_{\chi}(m;Q_1/p^{\eta}) A(m;p^{\eta}),
\end{equation}
where $A(m;p^{\eta})\ll m^{\varepsilon}$ by the divisor bound. The key thing is that these noisy coefficients do not depend on $\chi$, which will allow us to remove them via an application of the Cauchy--Schwarz inequality in section~\ref{moments-sec}. Combining \eqref{reduced range} through
\eqref{K alpha}
we obtain Proposition~\ref{second moment bound}.

\begin{prop}
\label{second moment bound}
Let $q_1$ and $q$ be subject to the conditions in \eqref{q restrictions}. Then there exist coefficients $A(m;p^{\eta})\ll m^{\varepsilon}$ such that, for every primitive character $\chi\Mod{q}$,
\[
S_2(\chi) \ll q^{\varepsilon} q_1 \bigg( 1 + Q_1^{-1/2} \Real \sum_{p^{\eta}\ll q^{1/2+ \varepsilon}q_1^{-1}}p^{\eta/2} \sumprime_{\substack{|m| \ll q^{1 + \varepsilon} (q_1^{2}p^{2\eta})^{-1}}}  K_{\chi}(m;Q_1/p^{\eta})
A(m;p^{\eta}) \bigg),
\]
where $K_{\chi}(m;Q_1/p^{\eta})$ are as in \eqref{Kchidef}.
\end{prop}

\section{Exponential sum estimates}
\label{expsums-sec}

In this section, we evaluate and estimate complete exponential sums to prime power moduli. Our principal tools are the $p$-adic stationary phase method Lemmata~\ref{stationary phase} and \ref{exact stationary phase}. In Lemma~\ref{Kchi-lemma}, we consider the complete exponential sum $K_{\chi}(j,h;\tilde{q})$ introduced in \eqref{Kchidef} and show that it can be expressed in terms of explicit exponentials $K_{\chi}^{\pm}(m;\tilde{q})$ with $p$-adically analytic phases. Then, in Proposition~\ref{bound for completion}, we show square-root cancellation in complete sums of products of $K_{\chi}^{\pm}(m;\tilde{q})$ including additive twists.

\subsection{Evaluation of \texorpdfstring{$K_{\chi}(m;\tilde{q})$}{K-chi(m,q)}}

In the following lemma, we explicitly evaluate the complete exponential sum $K_{\chi}(j,h;\tilde{q})$.

\begin{lem}
\label{Kchi-lemma}
Let $\tilde{q}$ be a proper divisor of $q$ and, for every $j\in\mathbb{Z}$, let $p^{\eta_j}=(j,\tilde{q})$. Then, the sum $K_{\chi}(j,h;\tilde{q})$ defined in \eqref{Kchidef} satisfies
\begin{equation}
\label{Kchi-reduction}
K_{\chi}(j,h;\tilde{q})=\begin{cases} p^{\eta/2} K_{\chi}(jh/p^{2\eta};\tilde{q}/p^{\eta}),
&\eta=\eta_j=\eta_h, \, \tilde{q} \neq p^{\eta} ;\\
\tilde{q}^{1/2} (1-p^{-1}), & \eta_j=\eta_h =\eta_{\tilde{q}}; \\
-\tilde{q}^{1/2}/p, &\eta_j + \eta_h = 2\eta_{\tilde{q}} - 1;\\
0, & \text{otherwise}.\end{cases}
\end{equation}
Further, let $A$ be an integer such that \eqref{Postnikov formula} holds for $\chi$, and assume that $\tilde{q}\geqslant p^2$. Then, for $(m,\tilde{q})=1$,
\begin{equation}
\label{Kchi-pm}
K_{\chi}(m;\tilde{q})= K^+_{\chi}(m ; \tilde{q}) + K^-_{\chi}(m ; \tilde{q}),
\end{equation}
where
\[ K_{\chi}^{\pm}(m;\tilde{q}) = \begin{cases} \Delta_{\theta}(s_{\pm}(m/A;\tilde{q});\tilde{q}) e \left(Ag_{\pm}(m/A;\tilde{q})/\tilde{q}\right), & \big(\frac{Am}p\big)=1;\\ 0, &\text{otherwise},
\end{cases} \]
where, for $\big( \frac{m}{p} \big)=1$,
\begin{equation}
\label{def-g-spm}
\begin{aligned}
&g_{\pm}(m;\tilde{q})=(\tilde{q}/q)\log_p\big(1+(q/\tilde{q})s_{\pm}(m;\tilde{q})\big) +  m/s_{\pm}(m;\tilde{q}),\\
&s_{\pm}(m;\tilde{q})=\tfrac12\big(mq/\tilde{q}\pm\sqrt{(mq/\tilde{q})^2+4m}\big),
\end{aligned}
\end{equation}
$\theta$ is the phase associated to $K_{\chi}(m;\tilde{q})$ in \eqref{Kchi-m}, and $\Delta_{\theta}$ is as described in Lemma~\ref{exact stationary phase}.
\end{lem}

\begin{proof}
We first establish \eqref{Kchi-reduction}.
Write $h\equiv h' p^{\eta_h}\pmod{\tilde{q}}$ and $j\equiv j' p^{\eta_j}\pmod{\tilde{q}}$ where $p \nmid h'j'$, and set $\eta = \min \{ \eta_j, \eta_h \}$. If $p^{\eta} \neq \tilde{q}$, then by the substitution $r \mapsto \overline{j' r}$ for the variable of summation in \eqref{Kchidef} and a reduction to a sum over residues modulo $\tilde{q}/p^{\eta}$
we have
\begin{equation}
\label{Kchi-reduced}
    K_{\chi}(j,h;\tilde{q}) = \tilde{q}^{-1/2} p^{\eta} \sumprime_{r \Mod{\tilde{q}/p^{\eta}}} \chi \left( 1 + \frac{q}{\tilde{q}/p^{\eta}}p^{\eta_h-\eta}rj'h'\right) e \left( \frac{p^{\eta_j - \eta} \overline{r}}{\tilde{q}/p^{\eta}} \right).
\end{equation}
In particular, this proves the first case of \eqref{Kchi-reduction}. The second case, when $p^{\eta}=\tilde{q}$, follows from a trivial evaluation of the definition \eqref{Kchidef}.

We now assume $\eta_j \neq \eta_h$. For $\eta_j + \eta_h = 2 \eta_{\tilde{q}} - 1$, the situation quickly boils down, directly or with an application of \eqref{Postnikov formula}, to
\[
K_{\chi}(j,h;\tilde{q}) =
\tilde{q}^{-1/2}(\tilde{q}/p)\sumprime_{r \Mod{p}} e (r/p) = - \tilde{q}^{1/2}/p.
\]
In any other event, let $\phi$ be the phase associated to \eqref{Kchi-reduced} where
\begin{equation}
\label{phase phi}
    \phi^{(k)}(x) = (-1)^{k-1} (k-1)! \left( A \left( \frac{p^{\eta_h-\eta}j'h'}{1+x j'hq/\tilde{q}} \right)^k  \left(\frac{q}{\tilde{q}/p^{\eta}}\right)^{k-1} - kp^{\eta_j - \eta} x^{-(k+1)} \right)
\end{equation}
for $k \geqslant 1$ and $x\in\mathbb{Z}_p^{\times}$ by \eqref{Postnikov formula}. From \eqref{phase phi}, it is easily seen that $\phi^{(k)}(x)\in\mathbb{Z}_p$ for $x \in \Z_p^{\times}$, and the classical bound $\ord_p(k!)\leqslant k/(p-1)$ shows that
$\rt^*(\tilde{q}/p^{\eta})^k \phi^{(k)}(x)/k! \equiv 0 \Mod{\tilde{q}/p^{\eta}}$
for all $k\geqslant 2$ and $x \in \Z_p^{\times}$. Thus $\phi$ satisfies the hypotheses of Lemma~\ref{stationary phase} with $p^{\ell}=\rt^*(\tilde{q}/p^{\eta})$. Since in this case $\rt_*(\tilde{q}/p^{\eta}) \geqslant p$ and exactly one of $\eta_h$ and $\eta_j$ equals $\eta$, we find that $K_{\chi}(j,h;\tilde{q})$ must vanish since no solutions to $\phi'(x) \equiv 0 \Mod{p}$ exist in $(\Z/p\Z)^{\times}$. This completes the proof of \eqref{Kchi-reduction}.

Next, for $A$, $(m,\tilde{q})=1$, and $\tilde{q} \geqslant p^2$ as stated,
the phase $\theta$ associated to $K_{\chi}(m;\tilde{q})$ in \eqref{Kchi-m} agrees with the phase $\phi$ in \eqref{Kchi-reduced} and \eqref{phase phi} with $h=1$, $j=j'=m$, and $\eta_j=\eta_h=\eta=0$  (substituting $r\mapsto\overline{m}r$); in particular,
\[ \theta'(x) = \frac{A}{1+xq/\tilde{q}} - \frac{m}{x^2} \quad \text{and} \quad \theta''(x) = -\frac{Aq/\tilde{q}}{(1+xq/\tilde{q})^2} + \frac{2m}{x^3}. \]
We will use Lemma~\ref{exact stationary phase}.
If $\big(\frac{Am}p\big)=-1$, there are no solutions to $\theta'(x_0) \equiv 0 \Mod{\rt_*(\tilde{q})}$ by an obstruction modulo $p$, so that $K_{\chi}(m;\tilde{q})=0$ in this case. Otherwise, solving the equation $\theta'(x_0)=0$ yields $x_0 = s_{\pm}(m/A;\tilde{q})$. Upon noting that $\theta''(x_0)\equiv 2mx_0^{-3} \not \equiv 0 \Mod{p}$, an application of Hensel's lemma gives exactly two unique solutions to the congruence above, proving \eqref{Kchi-pm}.
\end{proof}

\subsection{Sums of products}
As we input Proposition~\ref{second moment bound} into estimating short second moments in aggregate over sets of characters, we will incur incomplete sums of products of trace functions $K_{\chi}(m;\tilde{q})$ evaluated in Lemma~\ref{Kchi-lemma}, with two different characters $\chi$. Specifically, the inner sum in \eqref{Tbound} will be estimated using the method of completion, Lemma~\ref{completion}. In preparation for this, in this section we prove the following proposition.

\begin{prop}
\label{bound for completion}
Let $\tilde{q}\geqslant p^2$ be a proper divisor of $q$ and $K_{\chi}(m;\tilde{q})$ be as in \eqref{Kchidef}. Further, let $\chi$ and $\chi'$ be two primitive Dirichlet characters modulo $q$ with associated units $A$ and $A'$ as in \eqref{Postnikov formula}. Denote $\delta_q(\chi,\chi') = (q/p,A-A')$ and $Q=\tilde{q}/(\tilde{q},\delta_q(\chi,\chi'))$. Then:
\begin{enumerate}
\item for $Q\geqslant p$, the expression $K_{\chi}^{\pm}(m;\tilde{q})\overline{K_{\chi'}^{\pm}(m;\tilde{q})}$ is $Q$-periodic and satisfies, for every $v\in\mathbb{Z}$,
\begin{equation}
\label{completion squareroot}
\sumprime_{u \Mod{Q}} K^{\pm}_{\chi} (u;\tilde{q}) \overline{K^{\pm}_{\chi'} (u;\tilde{q})} e \left( \frac{-uv}{Q} \right) \ll Q^{1/2};
\end{equation}
\item for $Q\geqslant p^2$, the left-hand side of \eqref{completion squareroot} vanishes unless $|v|_p=1$;
\item\label{Q1} for $Q=1$, $K_{\chi}^{\pm}(m;\tilde{q})\overline{K_{\chi'}^{\pm}(m;\tilde{q})}=\mathbf{1}_{(Am/p)=1}$.
\end{enumerate}
\end{prop}

\begin{proof}
By Lemma~\ref{Kchi-lemma}, the sum on the left-hand side of \eqref{completion squareroot} vanishes unless $AA'\in\Z_p^{\times 2}$; we assume this henceforth and restrict the sum (as we may) to $(\frac{u/A}p)=1$. Further, let $\theta_{\chi}$ and $\theta_{\chi'}$ be phases associated to $K_{\chi}(m;\tilde{q})$ and $K_{\chi'}(m;\tilde{q})$ as in \eqref{Kchi-m}, respectively. Then, by Lemma~\ref{Kchi-lemma}, we have for $(\frac{m/A}p)=1$
\begin{align*}
& K_{\chi}^{\pm}(m;\tilde{q}) \overline{K_{\chi'}^{\pm}(m;\tilde{q})}\\
&\qquad = \Delta_{\theta_{\chi}} (s_{\pm}(m/A;\tilde{q});\tilde{q}) \overline{ \Delta_{\theta_{\chi'}} (s_{\pm}(m/A';\tilde{q});\tilde{q}) } e \left( \frac{ A g_{\pm}(m/A;\tilde{q}) - A' g_{\pm}(m/A';\tilde{q}) }{(\tilde{q},\delta_q(\chi,\chi')) Q} \right)
\end{align*}
where, for fixed $\chi$ and $\chi'$, the product of the $\Delta$ factors depends only on the class of $m \Mod{p}$, as is readily verified using $\theta_{\chi}'(x_{0,A})=0$ and $\theta_{\chi}''(x_{0,A})\equiv 2mx_{0,A}^{-3}\Mod p$ with $x_{0,A}=s_{\pm}(m/A;\tilde{q})$ from the proof of Lemma~\ref{Kchi-lemma}. This proof also shows that $g_{\pm}(m;\tilde{q})$, a function analytic on its domain $\Z_p^{\times 2}$, is invariant modulo $\tilde{q}$ under translation by $\tilde{q}\Z_p$. Moreover, a moment's reflection on the definition \eqref{def-g-spm} combined with \eqref{sqrt-mult} and \eqref{sqrt-expansion} shows that for $m\in\Z_p^{\times 2}$ both $s_{\pm}(m;\tilde{q})$ and $g_{\pm}(m;\tilde{q})$ may be expanded into a convergent power series in $\sqrt{m}$ with coefficients in $\mathbb{Z}_p$. From this and \eqref{ord-sqrt} it follows that, for $m\in A\Z_p^{\times 2}$,
\[
A g_{\pm}(m/A;\tilde{q}) - A' g_{\pm}(m/A';\tilde{q}) = (\tilde{q}, \delta_q(\chi,\chi')) \cdot \sigma(m)
\]
is divisible by $(\tilde{q}, \delta_q(\chi,\chi'))$ and invariant modulo $\tilde{q}$ under translation by $Q\Z_p$ (for $Q\geqslant p$). This establishes the periodicity claim and \eqref{Q1} follows from the definition of $\Delta_{\theta}$, noting that $A \equiv A' \Mod{p}$ in this case.

As for the estimate \eqref{completion squareroot}, the case $Q=p$ is trivial, so we assume that $Q \geqslant p^2$. We will be interested in applying Lemma~\ref{stationary phase} and Remark~\ref{invariance} with $p^n = Q$, phase $\sigma$, and $p^{\ell}=\rt^*(Q)$. Since $s_{\pm}(u/A; \tilde{q})$ solves $\theta'(x_0)=0$ in the proof of Lemma~\ref{Kchi-lemma}, we observe
\[ \frac{\dx}{\dx u} A g_{\pm} (u/A;\tilde{q}) = \left\{ \frac{\partial}{\partial u} + \frac{\partial}{\partial s_{\pm}}\cdot\frac{\dx s_{\pm}(u/A; \tilde{q})}{\dx u} \right\} A g_{\pm} (u/A;\tilde{q}) = \frac{1}{s_{\pm}(u/A ; \tilde{q})}, \]
so that, by rationalizing denominators,
\[ \sigma'(u) = \frac{\pm1/2}{\delta_q(\chi,\chi')}\left( \sqrt{ (q/\tilde{q})^2  + 4A/u } - \sqrt{ (q/\tilde{q})^2  + 4A'/u } \right) - v. \]
Expanding the difference of roots above according to \eqref{sqrt-mult} and \eqref{sqrt-expansion} yields the quantity
\[ \sum_{k \geqslant 0} \binom{1/2}{k} (q/\tilde{q})^{2k} \big( (4A/u)^{1/2} (u/4A)^k - (4A'/u)^{1/2} (u/4A')^k \big), \]
which, along with \eqref{ord-sqrt}, shows that the sum in \eqref{completion squareroot} with phase $\sigma$ satisfies the appropriate conditions in Lemma~\ref{stationary phase} (keeping in mind Remark \ref{invariance}). From here and \eqref{ord-sqrt}, we see that the sum in \eqref{completion squareroot} vanishes unless $|v|_p = 1$, as solutions to the stationary phase congruence $\sigma'(u) \equiv 0 \Mod{\rt_*(Q)}$ could not exist otherwise. Since $\alpha^{1/2} \beta^{1/2}/(\alpha \beta)^{1/2}\in\{\pm 1\}$  for every $\alpha,\beta \in \Z_p^{\times 2}$, any solutions to the stationary phase congruence must satisfy one of the four congruences
\[ \delta_{q}(\chi,\chi')^{-1} \sum_{k \geqslant 0} \binom{1/2}{k} (q/\tilde{q})^{2k} \big( \epsilon_1 \left( u/4A \right)^k -\epsilon_2 (A'/A)^{1/2} (u/4A')^k \big) \equiv v \left(  u/A \right)^{1/2} \Mod{\rt_*(Q)} \]
with $\epsilon_i\in\{\pm 1\}$, where in fact $\epsilon_1=\epsilon_2$ unless, possibly, $\delta_q(\chi,\chi') = 1$.
Each of these four congruences is polynomial in $(u/A)^{1/2}$ modulo $\rt_*(Q)$, satisfies the hypotheses of Hensel's lemma, and reduces to a non-degenerate linear congruence in $(u/A)^{1/2}$ modulo $p$. Thus there are $O(1)$ solutions modulo $\rt_*(Q)$ to the stationary phase congruence. The proposition then follows.
\end{proof}

\section{Short second moment estimates}
\label{moments-sec}
Proposition~\ref{second moment bound} provides an individual bound for the short second moment $S_2(\chi)$ in terms of averages of the arithmetic function $K_{\chi}(m;\tilde{q})$. In this section, we leverage this result and the estimates on exponential sums from section~\ref{expsums-sec} to prove in Proposition~\ref{short moment bound} our penultimate result, an aggregate bound on the short second moment over a \emph{collection} of characters modulo $q_2$.

\begin{prop}
\label{short moment bound}
Let $q_1$, $q_2$, and $q$ be subject to the conditions in \eqref{q restrictions}.
Let $\chi$ be any primitive character modulo $q$, and let $\Psi$ be any set of Dirichlet characters modulo $q_2$. Then
\[ \sum_{\psi_2 \in \Psi} S_2 (\chi \psi_2) \ll q^{\varepsilon} \left(  \big(q_1 + q_1^{1/4}q_2^{1/4} \big) |\Psi| + q^{1/2} |\Psi|^{1/2} \right). \]
\end{prop}

\begin{proof}
We will use Proposition~\ref{second moment bound}; in its notation, we may assume that $Q_1/p^{\eta}\gg q^{1/2-\varepsilon}\geqslant p^2$, as Proposition~\ref{short moment bound} is trivially true for $q\ll p^4$. Decomposing $K_{\chi}(m;Q_1/p^{\eta})$ as $K_{\chi}^{+}+K_{\chi}^{-}$ as in Lemma~\ref{Kchi-lemma}, Proposition~\ref{second moment bound} gives us
\begin{equation}
\label{short second moment}
    \sum_{\psi_2 \in \Psi} S_2 (\chi \psi_2) \ll q^{\varepsilon} \left(  q_1 |\Psi| + q_1^{3/2} q^{-1/2} \big( T^+(\Psi) + T^-(\Psi) \big) \right),
\end{equation}
where   
\[
T^{\pm}(\Psi) := \Real \sum_{\psi_2 \in \Psi}\sum_{p^{\eta}\ll q^{1/2+\varepsilon}q_1^{-1}} p^{\eta/2}\sumprime_{\substack{|m| \ll q^{1 + \varepsilon} (q_1^{2} p^{2\eta})^{-1} }} K_{\chi \psi_2}^{
    \pm} (m;Q_1/p^{\eta}) A(m;p^{\eta})
\]
and $A(m;p^{\eta})\ll m^{\varepsilon}$. Application of the Cauchy--Schwarz inequality produces the bound
\begin{equation}
\label{Tbound}
    T^{\pm}(\Psi) \leqslant \frac{q^{1/2 + \varepsilon}}{q_1}
    \Bigg( \sum_{\psi_2,\psi_2' \in \Psi} 
     \sum_{p^{\eta} \ll q^{1/2 + \varepsilon} q_1^{-1}}p^{-\eta}
    \sumprime_{\substack{
    |m| \ll q^{1 + \varepsilon} (q_1^{2}p^{2\eta})^{-1} }} K_{\chi \psi_2}^{\pm}(m;Q_1/p^{\eta}) \overline{K_{\chi \psi_2'}^{\pm}(m;Q_1/p^{\eta})} \Bigg)^{1/2}.
\end{equation}
By Proposition~\ref{bound for completion}, the inner summand above is periodic modulo
\[
Q_{\eta,\psi_2,\psi_2'} = \frac{Q_1/p^{\eta}}{(Q_1/p^{\eta},(q/q_2)\delta_{q_2}(\psi_2,\psi_2'))}
\]
whenever $Q_{\eta,\psi_2,\psi_2'} \geqslant p^2$. Moreover, in this case, any complete segments of the inner sum in \eqref{Tbound} vanish, and an application of Lemma~\ref{completion} and Proposition~\ref{bound for completion} to any remaining incomplete segment gives that
\[    \sumprime_{\substack{|m| \ll q^{1 + \varepsilon} (q_1^{2}p^{2\eta})^{-1} }} K_{\chi \psi_2}^{\pm}(m;Q_1/p^{\eta}) \overline{K_{\chi \psi_2'}^{\pm}(m;Q_1/p^{\eta})} \ll q^{\varepsilon}
Q_{\eta,\psi_2,\psi_2'}^{1/2}.
\]
When $Q_{\eta,\psi_2,\psi_2'} \leqslant p$, we bound the inner sum of \eqref{Tbound} trivially. The second to inner-most sum of \eqref{Tbound} is therefore asymptotically bounded above by
\begin{align*}
    &\sum_{\substack{p^{\eta} \ll q^{1/2 + \varepsilon}q_1^{-1} \\ Q_{\eta,\psi_2,\psi_2'}\geqslant p^2}}  \frac{q^{\varepsilon} Q_{\eta,\psi_2,\psi_2'}^{1/2}}{p^{\eta}} + \sum_{\substack{p^{\eta} \ll q^{1/2 + \varepsilon}q_1^{-1} \\ Q_{\eta,\psi_2,\psi_2'} \leqslant p}} \frac{q^{1 + \varepsilon}}{q_1^2 p^{3\eta}}\\
    &\qquad\qquad \ll q^{\varepsilon} \left( \frac{q_2}{q_1
    \delta_{q_2}(\psi_2,\psi_2')} \right)^{1/2} + q^{1+\varepsilon}\min\left(\frac{q_1\delta_{q_2}(\psi_2,\psi_2')^{3}}{q_2^{3}},\frac1{q_1^2}\right),
\end{align*}
where in fact the second term only enters if $\delta_{q_2}(\psi_2,\psi_2') \gg q_2 q^{-(1/2 + \varepsilon)}$. Inserting the above into \eqref{Tbound}, we obtain
\[    T^{\pm}(\Psi) \ll \frac{q^{1/2 + \varepsilon}}{q_1} \Bigg(  \frac{q_2^{1/4}}{q_1^{1/4}} |\Psi| + \frac{q^{1/2}}{q_1^{1/2}} |\Psi|^{1/2} \Bigg). \]
Inserting this bound into \eqref{short second moment},
we complete the proof of Proposition~\ref{short moment bound}.
\end{proof}

\section{Proof of Theorem~\ref{R bound}}
\label{mainproof_sec}
Let $\varepsilon$ be an arbitrary, small positive real number. We will begin by defining the sets
\begin{align*}
    R_2&(V;\chi) := \{ \chi_2 \Mod{q_2} : \chi \chi_2 \in R(V;q) \} \\
    &\text{and} \quad \Psi(V;\chi) := \{ \psi_2 \Mod{q_2} : \psi_1 \psi_2 \in R_2(V;\chi) \text{ for some } \psi_1 \Mod{q_1}\}.
\end{align*}

Clearly we may assume that $q$ is a sufficiently high power of $p$.
Asymptotics of the fourth moment of Dirichlet $L$-functions due to Heath-Brown~\cite{Heath-Brown1981b} and later improved by Soundararajan~\cite{Soundararajan2007} imply that
\[
\sum_{\chi \Mod{q}} \left|L\left(\tfrac{1}{2},\chi\right)\right|^4 \ll q^{1 + \varepsilon},
\]
from which it follows that $|R(V;q)| \ll q^{1 + \varepsilon} V^{-4}$. This suffices to handle the case when $V \leqslant q^{1/8 + 2 \varepsilon}$. On the other hand, by the Weyl bound for Dirichlet $L$-functions to prime power moduli due to Postnikov~\cite{Postnikov1955} (see also~\cite{Milicevic2016}), $|R_2(V;\chi)|=0$ for $V\geqslant q^{1/6+\varepsilon}$.

Consider now the values of $q^{1/8+2\varepsilon}\leqslant V\leqslant q^{1/6+\varepsilon}$. We combine the observations
\begin{equation*}
    |R_2(V;\chi)| \leqslant \frac{1}{V^2 \varphi(q_1)} \sum_{\psi_2 \in \Psi(V;\chi)} S_2(\chi \psi_2) \quad \text{and} \quad |\Psi(V;\chi)| \leqslant \varphi(q_1)R_2(V;\chi),
\end{equation*}
with Proposition~\ref{short moment bound} and the choices $V^2 q^{-3\varepsilon} \leqslant q_1 \leqslant V^2q^{-2\varepsilon}$ and $q_2 = q_1^3$. With these choices, we obtain
\begin{align*}
|R_2(V;\chi)|&\ll \frac{q^{\varepsilon}}{V^2\varphi(q_1)}\Big(q_1\varphi(q_1)|R_2(V;\chi)|+q^{1/2}\varphi(q_1)^{1/2}|R_2(V;\chi)|^{1/2}\Big)\\
&\ll q^{-\varepsilon}|R_2(V;\chi)|+\frac{q^{1/2-\varepsilon}}{q_1^{3/2}}|R_2(V;\chi)|^{1/2},
\end{align*}
from which in turn it follows that
\[ |R_2(V;\chi)| \ll \frac{q^{1 + \varepsilon}}{q_1 V^4}. \]
As a consequence,
\[    |R(V;q)| = \frac1{q_2} \sumprime_{\chi\Mod q} |R_2(V;\chi)| \ll q^{2 + \varepsilon} V^{-12}, \]
which completes the proof of Theorem~\ref{R bound}, and hence of Theorem~\ref{main}.

\printbibliography

\end{document}